\documentclass[oneside,english]{amsart}
\usepackage[T1]{fontenc}
\usepackage[latin9]{inputenc}
\usepackage{amsthm}
\usepackage{amssymb}
\usepackage{graphicx}
\usepackage{esint}
\usepackage{amsfonts}
\usepackage{xr}
\usepackage{color}
\usepackage{graphicx,epsfig}
\usepackage{amsmath}

\makeatletter
\numberwithin{equation}{section}
\numberwithin{figure}{section}
\theoremstyle{plain}
\newtheorem{thm}{\protect\theoremname}[section]
  \theoremstyle{remark}
  \newtheorem{rem}[thm]{\protect\remarkname}
  \theoremstyle{plain}
  \newtheorem{prop}[thm]{\protect\propositionname}
  \theoremstyle{plain}
  \newtheorem{lem}[thm]{\protect\lemmaname}
  \theoremstyle{plain}
  \newtheorem{cor}[thm]{\protect\corollaryname}

\makeatother

\usepackage{babel}
  \providecommand{\corollaryname}{Corollary}
  \providecommand{\lemmaname}{Lemma}
  \providecommand{\propositionname}{Proposition}
  \providecommand{\remarkname}{Remark}
\providecommand{\theoremname}{Theorem}

\begin{document}

\title[Relative motility and invasion speed]{Travelling waves for diffusive and strongly competitive systems:
relative motility and invasion speed }

\author{Léo Girardin, Grégoire Nadin}

\email{leo.girardin@ens-cachan.fr}
\begin{abstract}
Our interest here is to find the invader in a two species, diffusive
and competitive Lotka \textendash{} Volterra system in the particular
case of travelling wave solutions. We investigate the role of diffusion
in homogeneous domains. We might expect a priori two different cases:
strong interspecific competition and weak interspecific competition.
In this paper, we study the first one and obtain a clear conclusion:
the invading species is, up to a fixed multiplicative constant, the
more diffusive one.
\end{abstract}

\maketitle

\section{Introduction}

Competitive reaction \textendash{} diffusion systems have been widely
studied in the last few years. This mathematical model is motivated
by numerous applications: ecology, chemistry, genetics, etc. In general,
the mathematical formulation of this problem is, for some spatial
domain $\Omega$ (non-necessarily bounded), some $n\in\mathbb{N}$
and some positive constants $\left(d_{i},r_{i},a_{i},k_{i,j}\right)_{i,j\in\left\{ 1,\dots,n\right\} }$:
\begin{equation}\label{eq:generic_pde}
\forall i\in\left\{ 1,\dots,n\right\}\quad \partial_{t}u_{i}=d_{i}\Delta_{x}u_{i}+u_{i}\left(r_{i}-a_{i}u_{i}-\sum_{j\neq i}k_{i,j}u_{j}\right) \text{ in } \Omega\times\left(0,+\infty\right)
\end{equation}
One tough question is how their solutions and, when it exists, the
long-time steady state, depend on the diffusion rates $\left(d_{i}\right)_{i\in\left\{ 1,\dots,n\right\} }$.
Asymptotically, how do the species (if we see this as continuous approximation
of a population dynamics problem) represented by the densities 
$\left(u_{i}\right)_{i\in\left\{ 1,\dots,n\right\} }$
share the domain $\Omega$? Basically, in the neighbourhood of any spatial point $x$,
two cases may occur: either only one species persists (exclusion case)
or two or more persist (coexistence case). In the exclusion case,
the only persistent species is called invading species. A priori,
all the parameters participate in the determination of this invader:
number of species $n$, heterogeneity of $\Omega$, boundedness of
$\Omega$, boundary conditions, intrinsic growth rates $\left(r_{i}\right)_{i\in\left\{ 1,\dots,n\right\} }$,
interspecific competition rates $\left(k_{i,j}\right)_{i,j\in\left\{ 1,\dots,n\right\} }$,
intraspecific competition rates $\left(a_{i}\right)_{i\in\left\{ 1,\dots,n\right\} }$
and of course diffusion rates $\left(d_{i}\right)_{i\in\left\{ 1,\dots,n\right\} }$. 

The dependency on diffusion rates is a very open general problem.
Previous works show clearly that a very general result is for the
moment inachievable and that we are due to consider in each study
a specific case for the other parameters of the problem. A key work
in this area is the paper by Dockery et al. \cite{Dockery}. They
proved that, when $\Omega$ is bounded, heterogeneous, with Neumann
boundary conditions and when $k_{i,j}=1$ for all $i,j\in\left\{ 1,\dots,n\right\} $,
the less motile species \textendash{} that is the one with the lower
diffusion rate \textendash{} is the invading species. Their result
relies deeply on the heterogeneity, the basic idea being that each
species loses the individuals trying to invade unfavorable areas while,
in favorable areas, the competition helps the more concentrated one,
that is the less diffusive one. 

We leave the extension of Dockery\textquoteright s result for different
$\left(k_{i,j}\right)_{i,j\in\left\{ 1,\dots,n\right\} }$ to others
and wonder if a similar result can be obtained in homogeneous domains
(bounded or not).

Actually, it is quite tough to guess heuristically what could happen
in homogeneous domains. Indeed, on one hand, the more diffusive species
might be able to ignore its competitors long enough and invade the
whole territory while eliminating the competitors slowly. On the other
hand, the more concentrated species \textendash{} that is the less
diffusive one \textendash{} might benefit from the maxim \textquotedblleft unity
is strength\textquotedblright{} and eliminate slowly the dispersed
competitors and, asymptotically, invade the domain. It is well-known
that diffusion tends to bring unexpected results. In any case, if
something can revert the invasion, we expect it to be the competition.
With this in mind, we decide to focus first on the infinite competition
limit which should amplify the effects of competition. 

Many papers limit their study to the case $n=2$ (and so will we) because
then the system becomes monotonic and is therefore
much simpler to study than the general case. We will not use 
the monotonicity explicitly but it will be the underlying mechanism
behind many results.

When $n=2$, the PDE system can be rewritten: 

\[
\left\{ \begin{matrix}\partial_{t}u=d_{1}\Delta_{x}u+u\left(r_{1}-a_{1}u-k_{1}v\right) & \mbox{in} & \Omega\times\left(0,+\infty\right)\\
\partial_{t}v=d_{2}\Delta_{x}v+v\left(r_{2}-a_{2}v-k_{2}u\right) & \mbox{in} & \Omega\times\left(0,+\infty\right)
\end{matrix}\right.
\]

When there is no diffusion at all, this system becomes an ODE system.
Then, the steady state $\left(u,v\right)=\left(0,1\right)$ (resp.
$\left(u,v\right)=\left(1,0\right)$) is stable when $\frac{k_{1}r_{2}}{r_{1}a_{2}}>1$
(resp. $\frac{k_{2}r_{1}}{r_{2}a_{1}}>1$), unstable when $\frac{k_{1}r_{2}}{r_{1}a_{2}}<1$
(resp. $\frac{k_{2}r_{1}}{r_{2}a_{1}}>1$). Our interest lies in the
bistable case and more precisely in the so-called \textquotedblleft weak 
competition case\textquotedblright{} where $\frac{k_{1}r_{2}}{r_{1}a_{2}}$ 
and $\frac{k_{2}r_{1}}{r_{2}a_{1}}$
are larger than $1$ but close to $1$. In the monostable case, only
one species is a \textquotedblleft strong\textquotedblright{} competitor.

The infinite competition limit ($k_{1}\to+\infty$ and $\frac{k_{1}}{k_{2}}$
constant) has been studied by Dancer et al. in 1999 in the case of
bounded domains with Neumann boundary conditions \cite{Dancer-1}
(they also investigated Dirichlet conditions five years later \cite{Dancer-2}).
They obtained a free boundary Stefan problem and, under regularity
assumptions, a spatial segregation with an explicit condition on the
interface. In 2007, Nakashima and Wakasa \cite{Nakashima} studied the generation 
of interfaces for such systems and obtained a similar free boundary condition.

It is worth mentioning that the spatial segregation in multi-dimensional
domains for elliptic PDE yields highly non-trivial issues. It can
be either approached as a free boundary problem (Dancer \cite{Dancer-1},
Quitalo \cite{Quitalo}) or as an optimal partition problem (Conti
\cite{Conti-1,Conti-2}), but in both cases it is really a problem
in itself, which requires additional assumptions on the initial conditions
and a lot of work.

Therefore, our interest goes to unbounded homogeneous domains. Reaction
\textendash{} diffusion studies in such domains usually conjecture
the existence of propagation fronts and, when their existence can
be rigorously proved, derive from them some information on the dynamics
of the system and the long-time steady state. Here, it is important
to recall that the main underlying assumption with propagation fronts
is that, when the initial conditions are well-chosen, the solutions
of the PDE asymptotically \textquotedblleft behave like\textquotedblright{}
the travelling wave solution. We refer to Gardner \cite{Gardner}
for such results for finite $k$. We will not treat this aspect of
the problem in this paper but will indeed investigate travelling wave
solutions.

A straightforward consequence of the travelling wave approach is that
it reduces the multi-dimensional $\Omega\times\mathbb{R}_{+}^{\star}$
to $\mathbb{R}$. The problem becomes one-dimensional, that is an
ODE problem, and thus all the free boundary issues vanish. Our hope
is to find a similar spatial segregation limit, with an explicit condition
on the interface connecting the invasion speed of the travelling
wave to the diffusion rates. We know from Gardner \cite{Gardner}
and Kan-On \cite{Kan-On} that the invasion speed is constant and
bounded by the Fisher \textendash{} KPP\textquoteright s speeds \cite{KPP}
of the species. Can we use the infinite competition limit to derive
its sign and therefore know which species invades the other? Will
unity be strength?

It is important to remark that the invasion speed is not linearly
determined here. Actually, a linearization near $\left(0,1\right)$
or $\left(1,0\right)$ yields no condition at all on the invasion
speed and the linearized speed cannot be defined as usual. 
As far as we know, the linear determinacy for competition \textendash{} 
diffusion systems is useful only with a specific
class of monostable problems (Huang \cite{Huang}, Lewis \cite{Lewis}).

In the next section, we fully expose the problem, enunciate our final
result and recall that the problem is well-posed. The third and main
section is dedicated to a compactness result and the convergence to
a limit problem which is similar in many ways to the one Dancer et
al. obtained. Eventually, the last section explicits the relation
between the speed and the diffusion rates.

\section{Formulation of the problem and main theorem}

In this first section, we present the PDE problem studied in this
article, give its ecological interpretation and enunciate our main
result. We also check quickly that the problem is well-posed.

\subsection{Model}

\subsubsection{Reaction \textendash{} diffusion system}

We first consider the following one-dimensional Lotka \textendash{}
Volterra competition \textendash{} diffusion problem:

\[
\left\{ \begin{matrix}\partial_{t}\mu=d_{1}\partial_{xx}\mu+\mu\left(r_{1}-a_{1}\mu-k_{1}\rho\right) & \mbox{in} & \mathbb{R}\times\left(0,+\infty\right)\\
\partial_{t}\rho=d_{2}\partial_{xx}\rho+\rho\left(r_{2}-a_{2}\rho-k_{2}\mu\right) & \mbox{in} & \mathbb{R}\times\left(0,+\infty\right)
\end{matrix}\right.
\]
 where $d_{1},d_{2},r_{1},r_{2},a_{1},a_{2},k_{1},k_{2}$ are positive
constants with ecological meaning (diffusion rates, intrinsic growth
rates, intraspecific competition rates, interspecific competition
rates). We assume, without loss of generality, that $\frac{k_{2}a_{2}}{r_{2}^{2}}\geq\frac{k_{1}a_{1}}{r_{1}^{2}}$. 

Let $k=\frac{k_{1}r_{2}}{a_{2}r_{1}}>0$, $\alpha=\frac{k_{2}a_{2}r_{1}}{k_{1}a_{1}r_{2}}>0$, 
$d=\frac{d_{2}}{d_{1}}>0$, $r=\frac{r_{2}}{r_{1}}>0$ and 
\[
\left(u_{k},v_{k}\right):\left(x,t\right)\mapsto\left(\frac{a_{1}}{r_{1}}\mu\left(\sqrt{\frac{d_{1}}{r_{1}}}x,\frac{1}{r_{1}}t\right),\frac{a_{2}}{r_{2}}\rho\left(\sqrt{\frac{d_{1}}{r_{1}}}x,\frac{1}{r_{1}}t\right)\right)
\]

We get:
\[
\left\{ \begin{matrix}\partial_{t}u_{k}=\partial_{xx}u_{k}+u_{k}\left(1-u_{k}\right)-ku_{k}v_{k} & \mbox{in} & \mathbb{R}\times\left(0,+\infty\right)\\
\partial_{t}v_{k}=d\partial_{xx}v_{k}+rv_{k}\left(1-v_{k}\right)-\alpha ku_{k}v_{k} & \mbox{in} & \mathbb{R}\times\left(0,+\infty\right)
\end{matrix}\right.
\]

As soon as $k>1$ (which will always be assumed thereafter), $\frac{\alpha k}{r}>1$, that is 
the system is bistable. Indeed, the free assumption 
$\frac{k_{2}a_{2}}{r_{2}^{2}}\geq\frac{k_{1}a_{1}}{r_{1}^{2}}$ we made earlier ensures that 
$\frac{\alpha}{r}\geq1$.

A priori, the parameters $k$, $\alpha$, $d$ and $r$ can take any positive value. 
Let $\mathcal{P}\left(k,\alpha,d,r\right)$ denote this generic PDE problem.
Our interest lies in the limit, as $k\to+\infty$, of the set of problems 
$\left\{\mathcal{P}\left(k,\alpha,d,r\right)\right\}_{k>1}$ 
(associated to a given $(\alpha,d,r)$) (hence the notations $u_k$ and $v_k$).  

Moreover, going back to the initial parameters, this means that we actually consider a larger class of
ecological problems than just $k_1\to+\infty$ and $\frac{k_1}{k_2}$ constant. 
Indeed, the only restrictions are that $\frac{d_2}{d_1}$, $\frac{r_2}{r_1}$ and 
$\frac{k_{2}a_{2}}{k_{1}a_{1}}$ are fixed along the whole class. 
For example, the limit $k\to+\infty$ may correspond to:
\begin{itemize}
    \item $k_2$ proportional (with a fixed constant along the whole class) to $k_1$ and $k_1\to+\infty$ with $a_1$ and $a_2$ fixed (along the whole class);
    \item $k_1\to+\infty$ and $a_1$ proportional to $\frac{1}{k_1}$ with $a_2$ and $k_{2}$ fixed;
    \item $a_2$ proportional to $a_1$ and $a_1\to0$ with $k_1$ and ${k_2}$ fixed.
\end{itemize}

\subsubsection{Travelling wave system}

Searching for a travelling wave of the variable $\xi=x-c_{k}t$, where
$c_{k}\in\mathbb{R}$ is the unknown invasion speed, the problem
rewrites eventually:

\begin{equation}\label{eq:ode_system}
\left\{ \begin{matrix}-u_{k}''-c_{k}u_{k}'=u_{k}\left(1-u_{k}\right)-ku_{k}v_{k} & \mbox{in} & \mathbb{R}\\
-dv_{k}''-c_{k}v_{k}'=rv_{k}\left(1-v_{k}\right)-\alpha ku_{k}v_{k} & \mbox{in} & \mathbb{R}\\
u_{k}\left(-\infty\right)=1,\ u_{k}\left(+\infty\right)=0\\
v_{k}\left(-\infty\right)=0,\ v_{k}\left(+\infty\right)=1\\
u_{k}'<0 & \mbox{in} & \mathbb{R}\\
v_{k}'>0 & \mbox{in} & \mathbb{R}
\end{matrix}\right.
\end{equation}

It is well-known that natural selection tends to differentiate the
niches of competing species. The travelling wave solution corresponds
to the case where $u_{k}$ lives essentially in the left half-space
while $v_{k}$ lives essentially in the right half-space. In such
a situation, it seems obvious that one species might chase the other
and invade the abandoned territory. The whole point of this article
is to determine this species, or equivalently, the sign of the invasion
speed. Indeed, 
\begin{enumerate}
\item $c_{k}>0$ iff $u_{k}$ chases $v_{k}$; 
\item $c_{k}<0$ iff $v_{k}$ chases $u_{k}$;
\end{enumerate}
Of course, we aim to find a result depending on the value of $d$.
Thus in the following pages, when we focus on the dependency of $c_{k}$
on $d$, we write $c_{k,d}$; otherwise, when $d$ is fixed, we simply
write $c_{k}$.

\subsection{\textquotedblleft Unity is not strength\textquotedblright{} theorem }

Our main result follows.
\begin{thm}
$\left(d\mapsto c_{k,d}\right)_{k>1}$ converges locally uniformly
in $\left(0,+\infty\right)$ to a continuous function $d\mapsto c_{\infty,d}$
which satisfies:
\begin{enumerate}
\item $c_{\infty,d}=0$ if $d=\frac{\alpha^{2}}{r}$;
\item $c_{\infty,d}\in\left(0,2\right)$ if $d\in\left(0,\frac{\alpha^{2}}{r}\right)$;
\item $c_{\infty,d}\in\left(-2\sqrt{rd},0\right)$ if $d>\frac{\alpha^{2}}{r}$.
\end{enumerate}
\end{thm}
\begin{rem}
This result is profoundly unexpected! It does not suffice to compare
$d$ to $1$ or $\alpha$ to $1$. $v$ can loose even if $r$ is
large and $u$ can loose even if $\alpha$ is large, for example.
This should yield interesting insight into ecological applications.
\end{rem}

\subsection{Well-posedness and regularity of the problem}
\begin{thm}
For any $k>1$, there exists a unique $c_{k}$ such that there exist
solutions $u_{k}$ and $v_{k}$ of the problem (\ref{eq:ode_system}). It is enforced that $c_{k}\in\left(-2\sqrt{rd},2\right)$,
$u_{k}\in\mathcal{C}^{\infty}\left(\mathbb{R}\right)$ and $v_{k}\in\mathcal{C}^{\infty}\left(\mathbb{R}\right)$.
We can moreover assume exactly one of the following normalization
hypotheses:
\begin{itemize}
\item $u_{k}\left(0\right)=v_{k}\left(0\right)$,
\item $u_{k}\left(0\right)=\frac{1}{2}$,
\item $v_{k}\left(0\right)=\frac{1}{2}$,
\end{itemize}
and if we do so, $u_{k}$ and $v_{k}$ are unique.\end{thm}
\begin{proof}
The well-posedness and the bounds for $c_{k}$ are proven by Gardner
in \cite{Gardner} and also by Kan-On in \cite{Kan-On} (actually,
Gardner only showed $c_{k}\in\left[-2\sqrt{rd},2\right]$ but Kan-On
showed indeed $c_{k}\in\left(-2\sqrt{rd},2\right)$ which will be
important in the end). It is worth mentioning that their papers actually
proved that the problem is well-posed without any monotonicity 
condition and that the monotonicity is indeed enforced.

Since $u_{k},v_{k}\in L^{\infty}\left(\mathbb{R}\right)$
and $u_{k}',v_{k}'\in L^{1}\left(\mathbb{R}\right)$, the regularity
just follows from $W^{k,p}$-estimates and Sobolev\textquoteright s
injections.\end{proof}
\begin{rem}
The extremal speeds $-2\sqrt{rd}$ and $2$ are the invasion speeds
of respectively $v_{k}$ when $u_{k}=0$ and $u_{k}$ when $v_{k}=0$.
This is a well-known result from Fisher, Kolmogorov, Petrovsky and
Piscounov \cite{KPP}.
\end{rem}

\section{Limit problem}

Here we show that $\left(u_{k}\right)$, $\left(v_{k}\right)$ and
$\left(c_{k}\right)$ converge when $k\to+\infty$ and formulate the
limit problem.

\subsection{Existence of limit points}

First, $\left(c_{k}\right)$ is relatively compact and therefore,
by the Bolzano \textendash{} Weierstrass theorem, has a limit point
$c\in\left[-2\sqrt{rd},2\right]$.

If $c\leq0$, we fix for any $k>1$ the normalization $u_{k}\left(0\right)=\frac{1}{2}$.
On the contrary, if $c>0$, we fix for any $k>1$ $v_{k}\left(0\right)=\left(\frac{1}{2}\right)$.
This choice will be explained later on. In either case, this implies
that the functions $k\mapsto u_{k}$ and $k\mapsto v_{k}$ are well-defined. 
\begin{prop}
For any $i\geq1$, let $K_{i}=\left[-i,i\right]$. $\left(u_{k}\right)$
and $\left(v_{k}\right)$ are relatively compact in $\mathcal{C}\left(K_{i}\right)$.\end{prop}
\begin{proof}
Our aim here is to use Ascoli\textquoteright s theorem. To that end,
let us show that each $u_{k}$ is Hölder-continuous with a constant
independent of $k$.

There exists a positive function $\chi\in\mathcal{D}\left(\mathbb{R}\right)$
such that $\chi\left(x\right)=0$ if $x\notin\left[-i-1,i+1\right]$
and $\chi\left(x\right)=1$ if $x\in\left[-i,i\right]$.

For any $k>1$, if we multiply the equation defining $u_{k}$ by $u_{k}\chi$
and then integrate, we get:
\[
\int\left(-u_{k}''u_{k}\chi-c_{k}u_{k}'u_{k}\chi\right)=\int u_{k}^{2}\chi-\int u_{k}^{2}\left(u_{k}+kv_{k}\right)\chi
\]

The third term is obviously negative. An integration by parts yields:
\[
\int u_{k}'^{2}\chi-\int\frac{u_{k}^{2}}{2}\chi''+c_{k}\int\frac{u_{k}^{2}}{2}\chi'\leq\int u_{k}^{2}\chi
\]

Finally, since $\int u_{k}'^{2}\chi\geq\int_{-i}^{i}u_{k}'^{2}$ and
$\|u_{k}\|_{L^{\infty}}\leq1$, we have:
\[
\|u_{k}'\|_{L^{2}\left(K_{i}\right)}^{2}\leq\int\left(\chi+\frac{\left|c_{k}\right|}{2}\left|\chi'\right|+\frac{1}{2}\left|\chi''\right|\right)
\]

Then we use Ascoli\textquoteright s theorem: the family $\left(u_{k}\right)$
is bounded in $L^{\infty}\left(K_{i}\right)$ and uniformly equicontinuous
in $K_{i}$ therefore it is relatively compact in $\mathcal{C}\left(K_{i}\right)$.
The exact same proof works for $\left(v_{k}\right)$.
\end{proof}
It is now clear, by a standard diagonal extraction argument, that
there exists a subsequence of $\left(u_{k}\right)$ (resp. $\left(v_{k}\right)$)
which converges locally uniformly to a limit point $u$ (resp. $v$).

\subsection{Properties of the limit points}

$c$, $u$ and $v$ are actually unique and true limits as it will
be proven later on. For the moment, let us just consider extracted
convergent subsequences, still denoted $\left(c_{k}\right)$, $\left(u_{k}\right)$
and $\left(v_{k}\right)$.
\begin{lem}
$uv=0$.\end{lem}
\begin{proof}
Multiplying by a test function $\varphi\in\mathcal{D}\left(\mathbb{R}\right)$
and integrating the equation for $u_{k}$ yields: 
\begin{eqnarray*}
k\left|\int u_{k}v_{k}\varphi\right| & \leq & \int u_{k}\left(1-u_{k}\right)\left|\varphi\right|+\left|c_{k}\right|\int u_{k}\left|\varphi'\right|+\int u_{k}\left|\varphi''\right|\\
 & \leq & C\|\varphi\|_{W^{2,1}\left(\mathbb{R}\right)}
\end{eqnarray*}

Hence $u_{k}v_{k}\to0$ in $\mathcal{D}'\left(\mathbb{R}\right)$. 

Since $u_{k}v_{k}\to uv$ locally uniformly, we get indeed $uv=0$.\end{proof}
\begin{rem}
This kind of result is usually referred to as a segregation property.
There is a lot of similar results in the literature.
\end{rem}
\begin{lem}
We have 
\[
-\alpha u''+dv''-\alpha cu'+cv'=\alpha u\left(1-u\right)-rv\left(1-v\right)
\]
 in $\mathcal{D}'\left(\mathbb{R}\right)$.\end{lem}
\begin{proof}
Multiply the equation for $u_{k}$ by $\alpha$ and substract to it
the one for $v_{k}$. The left-handside converges trivially in $\mathcal{D}'\left(\mathbb{R}\right)$.
The right-handside converges by dominated convergence.\end{proof}
\begin{lem}
$u,v\in C\left(\mathbb{R}\right)$ and $\alpha u-dv\in\mathcal{C}^{1}\left(\mathbb{R}\right)$.\end{lem}
\begin{proof}
The continuity of $u$ and $v$ is immediate thanks to the continuity
of each $u_{k}$ and $v_{k}$ and the locally uniform convergence.

Let $a,b\in\mathbb{R}$ such that $a<b$ and $I_{a}:C\left(\left[a,b\right]\right)\to C\left(\left[a,b\right]\right)$
defined by $I_{a}\left(f\right):x\mapsto\int_{a}^{x}f$. By continuity
of $u$ and $v$, it is quite obvious that the function 
\[
\alpha cu-cv+I_{a}\left(\alpha u\left(1-u\right)-rv\left(1-v\right)\right)-\left(\alpha cu\left(a\right)-cv\left(a\right)\right)
\]
 is continuous. But, thanks to the previous lemma, it is also equal
in $\mathcal{D}'\left(\left(a,b\right)\right)$ to $-\alpha u'+dv'$
up to an additive constant. Therefore $-\alpha u'+dv'$ is a well-defined
function of $\mathcal{C}\left(\left[a,b\right]\right)$.\end{proof}
\begin{lem}
$u$ and $v$ have finite limits at $\pm\infty$. Besides, 
\[
0\leq\lim_{+\infty}u\leq\lim_{-\infty}u\leq1
\]
 and 
\[
0\leq\lim_{-\infty}v\leq\lim_{+\infty}v\leq1
\]
\end{lem}
\begin{proof}
By locally uniform convergence, $u$ and $v$ are monotone, respectively
non-increasing and non-decreasing, and satisfy $0\leq u,v\leq1$. \end{proof}
\begin{lem}
$u$ and $v$ cannot vanish simultaneously on a non-empty compact
set.\end{lem}
\begin{proof}
Once again, we consider a non-empty compact set $[a,b]$. By monotonicity, if $u_{|\left[a,b\right]}=0$, then $u_{|[a,+\infty)}=0$.
Similarly, $v_{|(-\infty,b]}=0$. It yields, in $\mathcal{D}'\left(\left(-\infty,a\right)\right)$,
$-u''-cu'=u\left(1-u\right)$ and $\alpha u'-dv'=\alpha u'$. Therefore
$u'$ is continuous and, using $-u''-cu'=u\left(1-u\right)$, $u''$
is also continuous and the previous differential equation is satisfied
pointwise.

Now, we get by induction that $u$ is $\mathcal{C}^{\infty}$ in $\left(-\infty,a\right)$.
Since it does not explode on the left of $a$, it is the restriction
of a solution on a strictly larger interval. Since $u$ is regular,
$u'\left(a\right)=0$ and by Cauchy \textendash{} Lipschitz\textquoteright s
theorem, $u$ is identically null. By the same reasoning, $v$ is
also identically null. 

To prevent $u$ and $v$ from being both null on the whole real line,
either one of the two normalization sequences $\left(u_{k}\left(0\right)\right)_{k>1}=\left(\frac{1}{2}\right)$
and $\left(v_{k}\left(0\right)\right)=\left(\frac{1}{2}\right)$ combined
with locally uniform convergence suffices. \end{proof}
\begin{rem}
We already knew that $uv=0$ everywhere. Thus the previous lemma ensures that, 
for any $a<b$, $u_{|[a,b]}=v_{|[a,b]}=0$ is not possible; one of the two densities 
has to be positive whereas the other has to be null.
\end{rem}
\begin{lem}
Neither $u$ nor $v$ can be positive everywhere.\end{lem}
\begin{proof}
If $c\leq0$, the normalization sequence is $\left(u_{k}\left(0\right)\right)=\left(\frac{1}{2}\right)$.
It ensures that $u$ is not null. We define $\xi_{u}=\sup\left\{ \xi\in\mathbb{R}\ |\ u\left(\xi\right)>0\right\} \in(-\infty,+\infty]$. 

If $\xi_{u}=+\infty$ (that is, $u$ positive everywhere), $v$ is
null. 

In such a case, we have $u$ decreasing, bounded between $0$ and $1$, with limits
at infinity, non-constant by normalization, and $-u''-cu'=u\left(1-u\right)$
everywhere with $u\in\mathcal{C}^{\infty}\left(\mathbb{R}\right)$. 

This yields that $\lim_{-\infty}u=1$ and $\lim_{+\infty}u=0$. To
that end, we use L\textquoteright Hospital\textquoteright s rule. 

Let $l=u\left(-\infty\right)$, $G:\xi\mapsto\exp\left(c\xi\right)$
and $F=Gu'$ so that $F'=G\left(u''+cu'\right)=-Gu\left(1-u\right)$.
$F$ and $G$ are differentiable in $\mathbb{R}$, $G'\neq0$ and
$G\to+\infty$ as $\xi\to-\infty$; besides, $\frac{F'}{G'}\to-\frac{l\left(1-l\right)}{c}$.
By L\textquoteright Hospital\textquoteright s rule, $\frac{F}{G}\to-\frac{l\left(1-l\right)}{c}$,
that is $u'\left(-\infty\right)=-\frac{l\left(1-l\right)}{c}$. In
the end, necessarily, $l\in\left\{ 0,1\right\} $. 

At $+\infty$, we use the other version of L\textquoteright Hospital\textquoteright s
rule, noticing that $u'$ is bounded in $\mathbb{R}_{+}$ (easy to
prove) and checking that $F$ and $G$ go to $0$. Eventually, by
monotonicity, the limits are $1$ at $-\infty$ and $0$ at $+\infty$.

Thus $u$ is a travelling wave for the Fisher \textendash{} KPP equation
with speed $c\leq0<\sqrt{2}$, hence the contradiction \cite{KPP}.

If $c>0$, we just apply this reasoning to $v$ with normalization
$\left(v_{k}\left(0\right)\right)=\left(\frac{1}{2}\right)$. \end{proof}
\begin{cor}
The two quantities $\sup\left\{ \xi\in\mathbb{R}\ |\ u\left(\xi\right)>0\right\} $
and $\inf\left\{ \xi\in\mathbb{R}\ |\ v\left(\xi\right)>0\right\}$
are real and equal. Up to translation, we can assume it to be $0$.
By continuity of $u$ and $v$, $u\left(0\right)=v\left(0\right)=0$.\end{cor}
\begin{lem}
We have:
\begin{itemize}
\item $u\in\mathcal{C}^{\mathcal{1}}\left(\left(-\infty,0\right)\cup\left(0,+\infty\right)\right)$, 
\item $v\in\mathcal{C}^{\mathcal{1}}\left(\left(-\infty,0\right)\cup\left(0,+\infty\right)\right)$,
\end{itemize}
Besides, we can extend $u'$ and $v'$ by continuity on the left and
on the right respectively and obtain $u'\left(0\right)=\lim_{\xi\to0,\xi<0}u'\left(\xi\right)$
and $v'\left(\xi_{v}\right)=\lim_{\xi\to0,\xi>0}v'\left(\xi\right)$
which are finite and satisfy $-\alpha u'\left(0\right)=dv'\left(0\right)>0$.\end{lem}
\begin{proof}
$u$ is identically zero on $\left(0,\infty\right)$ so $u_{|\left(0,+\infty\right)}$
is trivially $\mathcal{C}^{\mathcal{1}}$. In $\left(-\infty,0\right)$,
it is a weak, and then regular (same routine), solution of $u''+cu'+u\left(1-u\right)=0$. 

Eventually, just recall that $\alpha u-dv\in\mathcal{C}^{1}\left(\mathbb{R}\right)$.
If its derivative at $0$ is zero, by the same kind of Cauchy \textendash{}
Lipschitz reasoning, $u=v=0$ everywhere.\end{proof}
\begin{rem}
The relation $\alpha u'\left(0\right) + dv'\left(0\right) = 0$ is 
essentially the free boundary condition obtained by Nakashima and 
Wakasa in \cite{Nakashima}.
\end{rem}
\begin{lem}
$\lim_{-\infty}u=1$ and $\lim_{+\infty}v=1$. \end{lem}
\begin{proof}
Same as before.\end{proof}
\begin{lem}
$c\in\left(-2\sqrt{rd},2\right)$, that is $c\notin\left\{ -2\sqrt{rd},2\right\} $.\end{lem}
\begin{proof}
Let us assume, for example, $c=-2\sqrt{rd}$. Let $\xi^{\star}>0$
such that $v\left(\xi^{\star}\right)=\frac{1}{2}$. 

We know from Fisher and KPP \cite{KPP} that $c=-2\sqrt{rd}$ is 
the maximal speed for wich there exists a travelling
wave $v_{KPP}$ positive, going from $0$ at $-\mathcal{1}$ to $1$
at $+\mathcal{1}$, which satisfies 
\[
-dv_{KPP}''-cv_{KPP}'=rv_{KPP}\left(1-v_{KPP}\right)
\]
We normalize by fixing $v_{KPP}\left(\xi^{\star}\right)=\frac{1}{2}$.
Let $f=v_{KPP}-v$. 

First, we can easily check that $f$ is in $\mathcal{C}\left(\mathbb{R}\right)\cap\mathcal{C}^{\infty}\left(\left(-\infty,0\right)\cup\left(0,+\infty\right)\right)$
and satisfies 
\[
-df''-cf'=rf\left(1-f\right)-2rvf
\]
 in $\left(0,+\infty\right)$. 

For any $\xi>\xi^{\star}$, $1-f\left(\xi\right)-2v\left(\xi\right)=1-v_{KPP}\left(\xi\right)-v\left(\xi\right)<0$,
with $f\left(\xi^{\star}\right)=0$. We can therefore apply the maximum
principle to the operator 
\[
d\bullet''+c\bullet'+r\left(1-f-2v\right)\bullet
\]
 in any interval $\left(\xi^{\star},b\right)$, $b>\xi^{\star}$.
Since $\lim_{+\infty}f=0$, it gives us that $f\left(\xi\right)\leq0$
for any $\xi\in\left(\xi^{\star},+\infty\right)$. But we can also
apply the minimum principle to the same operator, and we eventually
get that $f$ is identically zero in $\left(\xi^{\star},+\infty\right)$.
This way, $f'\left(\xi^{\star}\right)=0$, hence $f$ is identically
zero in $\left(0,+\infty\right)$, which is impossible since $f\left(0\right)>0$
and $f$ is continuous in $\mathbb{R}$.
\end{proof}

\subsection{Limit problem}

Let us sum up all these results in the following theorem.
\begin{thm}
There exist locally uniform limits $u$ and $v$ of $\left(u_{k}\right)$
and $\left(v_{k}\right)$ respectively. They satisfy:
\begin{enumerate}
\item $u,v\in\mathcal{C}\left(\mathbb{R}\right)\cap\mathcal{C}^{\infty}\left(\left(-\infty,0\right)\cup\left(0,+\infty\right)\right)$;
\item $\lim_{\xi\to-\infty}u\left(\xi\right)=1$;
\item $\lim_{\xi\to+\infty}v\left(\xi\right)=1$;
\item $u_{|\mathbb{R}_{+}}=0$;
\item $v_{|\mathbb{R}_{-}}=0$;
\item $u'\leq0$ in $\mathbb{R}_{-}$with $u'\left(0\right)$ defined by
left-continuity;
\item $v'\geq0$ in $\mathbb{R}_{+}$ with $v'\left(0\right)$ defined by
right-continuity;
\item $-u''-cu'=u\left(1-u\right)$ in $\left(-\infty,0\right)$;
\item $-dv''-cv'=rv\left(1-v\right)$ in $\left(0,+\infty\right)$;
\item $\alpha u'\left(0\right)=-dv'\left(0\right)$.
\end{enumerate}
\end{thm}

The behaviour of these limits is illustrated with the following figure.

\begin{figure}[h]\label{figure:thm}
\begin{center}
\resizebox{12cm}{!}{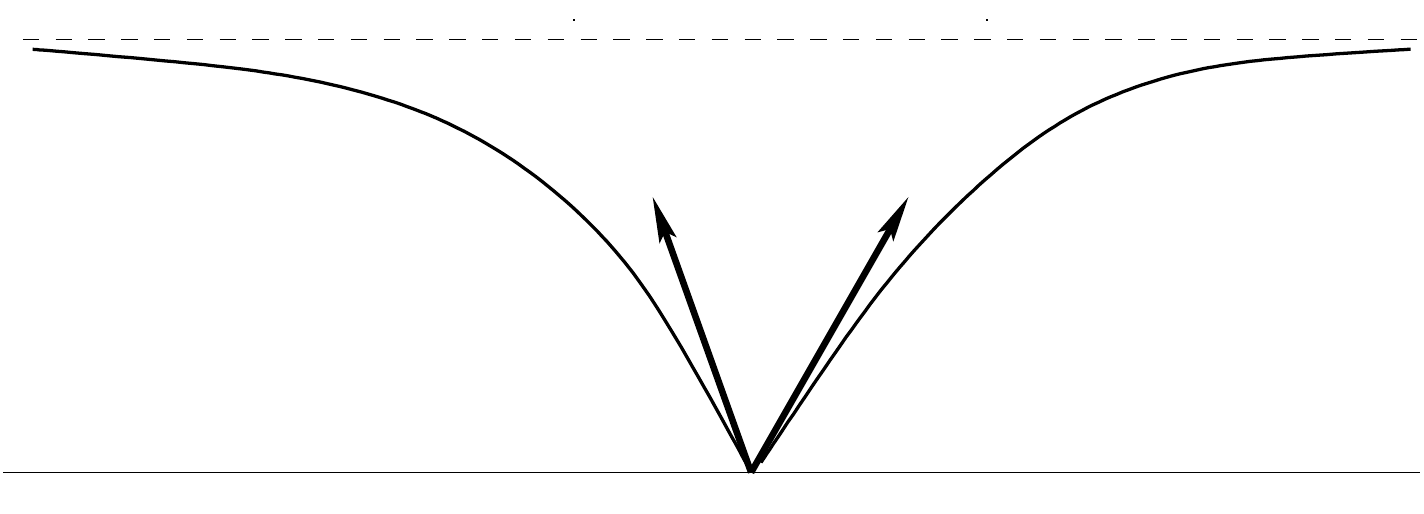}
\end{center}
\vspace{-.5cm}
\end{figure}

\subsection{Uniqueness of the limit points}
\begin{thm}
For any $c>-2$, the problem 
\[
\left\{ \begin{matrix}-y_{c}''-cy_{c}'=y_{c}\left(1-y_{c}\right) & \mbox{in} & \left(0,+\infty\right)\\
y_{c}\left(0\right)=0
\end{matrix}\right.
\]
 admits a unique positive solution. 

It satisfies $y_{c}'>0$ in $\mathbb{R}_{+}$ and $\lim_{\xi\to+\infty}y\left(\xi\right)=1$.
Besides, $\gamma:c\mapsto y_{c}'\left(0\right)$ is increasing and
continuous.\end{thm}
\begin{proof} This result was proved by Du and Lin in \cite{Du-1} (prop. 4.1) but wrongly stated. 
    Indeed, the requirement in their theorem should be $c<2$, not $c\geq 0$ as stated in \cite{Du-1}.
    (Moreover, be aware that our statement is written with $-cy_{c}'$ whereas their statement is 
    written 
    with $+cy_{c}'$; so the requirement $c<2$ becomes here $c>-2$; besides, this also
    changes the monotonicity of $\gamma$.)
    
Let us clear all doubts by filling the gap in their proof, that is the beginning where
they construct their subsolution. 

\begin{itemize}
    \item Case $|c|<2$: For all $\ell>0$, let $y^\ell$ the positive solution of 
$$\left\{ \begin{array}{ll}
-y''-cy' = y(1-y) &\hbox{ in } (0,\ell),\\
y(0)=y(\ell)=0.&\\
\end{array}\right.$$
According to Berestycki \cite{Berestycki} (theorem 4), such a solution exists if and only if the Dirichlet principal eigenvalue of the operator $-L$ on $(0,\ell)$ is negative: $\lambda_1 \big(-L, (0,\ell)\big)<0$, where $L$ is the operator associated with the linearized equation near $y=0$: 
$L\phi:= \phi''+c\phi'+\phi$. It is easy to compute: 
$$\lambda_1 \big(-L, (0,\ell)\big)= -1+c^2/4+\pi^2/\ell^2.$$
Hence, when $|c|<2$, one has $\lambda_1 \big(-L, (0,\ell)\big)<0$ when $\ell$ is large enough and 
thus we can construct $y^\ell$. Vice-versa, if $|c|\geq2$, $\lambda_1 \big(-L, (0,\ell)\big)>0$ and
the solution does not exist (whereas Du and Lin claim it does for all $c\leq0$).
    \item Case $c\geq2$: It suffices to remark that, for example, if $y_1$ is a solution of
the previous Dirichlet problem for some $c_1\in(-2,2)$, then $y_1$ is a subsolution for the 
Dirichlet problem with any speed $c>c_1$. 
\end{itemize}

In either case, the subsolution is now properly constructed and we can continue the proof as 
in \cite{Du-1} and conclude.

\end{proof} 

\begin{rem}
We need to change a bit $u$ and $v$ before pursuing in this direction.
Let us consider $\tilde{u}:\xi\mapsto u\left(-\xi\right)$ and $\tilde{v}:\xi\mapsto v\left(\sqrt{\frac{d}{r}}\xi\right)$.
$\tilde{u}$ is a solution of the problem 
\[
\left\{ \begin{matrix}-\tilde{u}''+c\tilde{u}'=\tilde{u}\left(1-\tilde{u}\right) & \mbox{in} & \left(0,+\infty\right)\\
\tilde{u}\left(0\right)=0
\end{matrix}\right.
\]

$\tilde{v}$ is a solution of the problem
\[
\left\{ \begin{matrix}-\tilde{v}''-\frac{c}{\sqrt{rd}}\tilde{v}'=\tilde{v}\left(1-\tilde{v}\right) & \mbox{in} & \left(0,+\infty\right)\\
\tilde{v}\left(0\right)=0
\end{matrix}\right.
\]

Besides, $c\in\left(-2\sqrt{rd},2\right)$ so $-c>-2$ and $\frac{c}{\sqrt{rd}}>-2$,
therefore we can apply the theorem.\end{rem}
\begin{cor}
For any $d>0$, there exists a unique $\left(u,v,c\right)$ satisfying
the limit problem (and may thereafter be called $\left(u_{\infty,d},v_{\infty,d},c_{\infty,d}\right)$). \end{cor}
\begin{proof}
The equality $-\alpha u'\left(0\right)=dv'\left(0\right)$ rewrites
$\alpha\gamma\left(-c\right)=\sqrt{rd}\gamma\left(\frac{c}{\sqrt{rd}}\right)$.
Now we consider the two functions $x\mapsto\alpha\gamma\left(-x\right)$
and $x\mapsto\sqrt{rd}\gamma\left(\frac{x}{\sqrt{rd}}\right)$. They
necessarily have an intersection point since $c$ exists. But as they
are respectively decreasing and increasing, this intersection point
is unique. 

The uniqueness of $c$ implies by the previous theorem the uniqueness
of $u$ and $v$. \end{proof}
\begin{cor}
The sequences $\left(c_{k}\right)$, $\left(u_{k}\right)$ and $\left(v_{k}\right)$
have a unique limit point each. Hence the pointwise convergence of
$\left(c_{k}\right)$ and locally uniform convergence of $\left(u_{k}\right)$
and $\left(v_{k}\right)$ are fully proved and there is no need to
consider extracted subsequences anymore.\end{cor}
\begin{proof}
Recall that, in any metric space, a sequence whose image is relatively
compact and which has a unique limit point converges to this limit
point.
\end{proof}
\begin{rem}
It is now clear that the sum up theorem of the previous section gives
sufficient but far from necessary conditions for uniqueness. For any
$c$, $u$ and $v$ are unique iff they are positive and satisfy points
4, 5, 8 and 9 and then the uniqueness of $c$ is just a consequence
of point 10. 
\end{rem}
\begin{prop}
The convergence of $\left(d\mapsto c_{k,d}\right)_{k>1}$ to $d\mapsto c_{\infty,d}$
is locally uniform.\end{prop}
\begin{proof}
Actually, one can see easily that the whole proof of pointwise convergence
of $\left(d\mapsto c_{k,d}\right)_{k>1}$ holds if we do not fix a
priori $d$. It suffices to have $d\in\left[D_{1},D_{2}\right]$,
with $D_{2}>D_{1}>0$ fixed, so that we can replace bounds like $-2\sqrt{rd}$
by $-2\sqrt{rD_{2}}$. 
\end{proof}

\section{Dependency of the invasion speed on the diffusion rates}

This last section is where we derive from the limit problem the result:
how does the invasion speed $c$ depend on the diffusion rate $d$?
Thanks to the convergence of $\left(c_{k}\right)$ to $c$, we will
then be able to extend it to $c_{k}$ (for $k$ large enough).
\begin{thm}
We have:
\begin{itemize}
\item if $d=\frac{\alpha^{2}}{r}$, $c_{\infty,d}=0$; 
\item if $d>\frac{\alpha^{2}}{r}$, $c_{\infty,d}\in\left(-2\sqrt{rd},0\right)$;
\item if $d<\frac{\alpha^{2}}{r}$, $c_{\infty,d}\in\left(0,2\right)$.
\end{itemize}
\end{thm}
\begin{proof}
The sign of $c_{\infty,d}$ is actually a simple consequence of the
relation $\alpha\gamma\left(-c\right)=\sqrt{rd}\gamma\left(\frac{c}{\sqrt{rd}}\right)$.
Indeed, let us prove that $rd<\alpha^{2}$ implies $c_{\infty,d}>0$.
Indeed, if $rd<\alpha^{2}$, then $\frac{\sqrt{rd}}{\alpha}<1$ and
as $\gamma\left(\frac{c}{\sqrt{rd}}\right)>0$, we get $\frac{\sqrt{rd}}{\alpha}\gamma\left(\frac{c}{\sqrt{rd}}\right)<\gamma\left(\frac{c}{\sqrt{rd}}\right)$.
Since $\gamma$ is increasing, $\frac{c}{\sqrt{rd}}>-c$, which clearly
implies that $c>0$. The case $rd>\alpha^{2}$ is similar.

If $rd=\alpha^{2}$, the relation becomes $\gamma\left(-c\right)=\gamma\left(\frac{c}{\sqrt{rd}}\right)$.
An obvious zero of $x\mapsto\gamma\left(-x\right)-\gamma\left(\frac{x}{\sqrt{rd}}\right)$
is $0$, and by monotonicity it is unique, hence $c=0$.\end{proof}
\begin{prop}
The function $d\mapsto c_{\infty,d}$ is continuous in $\left(0,+\infty\right)$.
\end{prop}
\begin{proof}
    This could follow from the continuity of each $d\mapsto c_{k,d}$
and the locally uniform convergence, but the continuity of $d\mapsto c_{k,d}$
is actually a more difficult problem (and is not solved by Kan-On
\cite{Kan-On}). Therefore, we prove the continuity of $d\mapsto c_{\infty,d}$
directly. Our proof being basically a repetition of the whole previous section of
this article, we give only a sketch of it. 

First, let $0<D_{1}<D_{2}$. We have: 
\begin{eqnarray*}
\left\{ c_{\infty,d}\ |\ d\in\left[D_{1},D_{2}\right]\right\}  & \subset & \left\{ c_{\infty,d}\ |\ d\in\left[D_{1},D_{2}\right]\cap\left(\frac{\alpha^{2}}{r},+\infty\right)\right\} \cup\left\{ 0\right\} \cup\left\{ c_{\infty,d}\ |\ d\in\cap\left(0,\frac{\alpha^{2}}{r}\right)\right\} \\
 & \subset & \left(\bigcup_{d\in\left[D_{1},D_{2}\right]\cap\left(\frac{\alpha^{2}}{r},+\infty\right)}\left[-2\sqrt{rd},0\right]\right)\cup\left[0,2\right]\\
 & \subset & \left[-2\sqrt{rD_{2}},2\right]
\end{eqnarray*}

Thus, $\left\{ c_{\infty,d}\ |\ d\in\left[D_{1},D_{2}\right]\right\} $
is a relatively compact subset of $\mathbb{R}$. 

Now, let $\delta\in\left[D_{1},D_{2}\right]$ and $\left(\delta_{n}\right)_{n\in\mathbb{N}}\in\left[D_{1},D_{2}\right]^{\mathbb{N}}$
a positive sequence which converges to $\delta$. Up to extraction,
$\left(c_{\infty,\delta_{n}}\right)$ converges to a limit point $C$. 

If $C\leq0$, we translate each couple $\left(u_{\infty,\delta_{n}},v_{\infty,\delta_{n}}\right)$
so that $\left(u_{\infty,\delta_{n}}\left(0\right)\right)=\left(\frac{1}{2}\right)$.
If $C>0$, we translate each couple $\left(u_{\infty,\delta_{n}},v_{\infty,\delta_{n}}\right)$
so that $\left(v_{\infty,\delta_{n}}\left(0\right)\right)=\left(\frac{1}{2}\right)$.
In either case, $\left\{ u_{\infty,d}\ |\ d\in\left[D_{1},D_{2}\right]\right\} $
and $\left\{ v_{\infty,d}\ |\ d\in\left[D_{1},D_{2}\right]\right\} $
are relatively compact in each $C\left(K_{i}\right)$ by Ascoli\textquoteright s
theorem, and, up to extraction, $\left(u_{\infty,\delta_{n}}\right)$
and $\left(v_{\infty,\delta_{n}}\right)$ converge locally uniformly.
Let $U$ and $V$ be their limits. 
\begin{itemize}
\item We have $-\alpha U''+\delta V''-\alpha CU'+CV'=\alpha U\left(1-U\right)-rV\left(1-V\right)$
in $\mathcal{D}'\left(\mathbb{R}\right)$.
\item $U$ and $V$ are continuous, $\alpha U-\delta V$ is $C^{1}$.
\item $U$ and $V$ are positive and have finite limits at infinity.
\item $UV=0$.
\item If $C\leq0$, $U$ is not identically null by normalization and $V$
cannot be identically null since if it was, $U$ would be a travelling
wave for the Fisher \textendash{} KPP equation with a speed smaller
than $2$. The same reasoning applies for $C>0$ and finally, neither
$U$ nor $V$ can be identically null.
\item $U$ and $V$ cannot be both null on a compact subset by continuity
of $\left(\alpha U-\delta V\right)'$ and a Cauchy \textendash{} Lipschitz\textquoteright s
argument. 
\end{itemize}
Now we translate back so that 
\[
\sup\left\{ \xi\in\mathbb{R}\ |\ U\left(\xi\right)>0\right\} =\inf\left\{ \xi\in\mathbb{R}\ |\ V\left(\xi\right)>0\right\} =0
\]

It yields $U_{|\mathbb{R}_{+}}=0$, $V_{|\mathbb{R}_{-}}=0$, $-U''-CU'=U\left(1-U\right)$
in $\left(-\infty,0\right)$, $-\delta V''-CV'=rV\left(1-V\right)$
in $\left(0,+\infty\right)$ and $\alpha U'\left(0\right)=-\delta V'\left(0\right)$.
Basically, $C$, $U$ and $V$ verify the exact same problem than
$c_{\infty,\delta}$, $u_{\infty,\delta}$ and $v_{\infty,\delta}$.
By uniqueness, $C=c_{\infty,\delta}$, that is $c_{\infty,\delta}$
is the unique limit point of $\left(c_{\infty,\delta_{n}}\right)$
and eventually $c_{\infty,\delta_{n}}\to c_{\infty,\delta}$. Therefore,
$d\mapsto c_{\infty,d}$ is indeed continuous.

\end{proof}

\section{Conclusion}
We have proved our \textquotedblleft Unity is not strength\textquotedblright{} theorem.
Some remaining questions concern the shape of the asymptotic speed: what are the 
limits when $d\to0$ or $d\to +\infty$? Are there optimal diffusion rates so that the invasion
of one species or the other is the fastest? And eventually, how fast is the convergence to this 
asymptotic limit and, for example, is it monotone?

These could be adressed with the knowledge of the derivatives of the speed as a function of 
$k$ or $d$. These might be
determined analytically thanks to Kan-On formulas \cite{Kan-On}.
However, we did not manage to compute the sign of these derivatives,
that is, the monotonicity of the speed with respect to $k$ or $d$.
We leave it as an open problem.

\end{document}